\documentclass[10pt]{article}
\usepackage{amsmath,amsthm,amsfonts,amssymb,latexsym,graphicx,etoolbox,stmaryrd,environ}
\usepackage[shortcuts]{extdash}

\allowdisplaybreaks[4]

\usepackage{color}
\usepackage[pdfpagemode=UseNone,pdfstartview=FitH]{hyperref}

\newtheorem{theorem}{Theorem}
\newtheorem{proposition}[theorem]{Proposition}

\theoremstyle{definition}

\theoremstyle{remark}
\newtheorem{remark}[theorem]{Remark}

\newcommand{\R}{\mathbb{R}}

\newcommand{\dd}{\,\mathrm{d}}
\renewcommand{\d}{\mathrm{d}}

\renewcommand{\P}{\mathbb{P}}
\newcommand{\E}{\mathbb{E}}
\newcommand{\FFF}{\mathcal{F}}

\title{Validity and efficiency of the conformal CUSUM procedure}
\author{Vladimir Vovk, Ilia Nouretdinov, and Alex Gammerman}

\begin{document}
\maketitle
\begin{abstract}
  In this paper we study the validity and efficiency of a conformal version
  of the CUSUM procedure for change detection both experimentally and theoretically.

  The version of this paper at \url{http://alrw.net} (Working Paper 41)
  is updated most often.
\end{abstract}

\section{Introduction}

One of the most popular methods for change detection
is Page's \cite{Page:1954} CUSUM procedure;
see \cite[Chap.~6]{Poor/Hadjiliadis:2009} for a modern exposition.
The CUSUM procedure has been shown to be optimal in some important senses,
including Lorden's \cite{Lorden:1971,Moustakides:1986} worst-case definition
and Ritov's \cite{Ritov:1990} game-theoretic definition.

The basic setting in which the optimality of CUSUM has been shown
is where the pre-change and post-change distributions are known,
and the only unknown is the time when the change happens (the changepoint).
The conformal CUSUM procedure,
introduced in \cite{Volkhonskiy/etal:2017COPA}
and discussed in detail in \cite[Chap.~8]{Vovk/etal:2022book},
only assumes that the pre-change data distribution is IID
(the pre-change observations are independent and identically distributed).

In \cite[Chap.~8]{Vovk/etal:2022book}
we only discussed the \emph{validity} of the conformal CUSUM procedure,
meaning the frequency of false alarms
(i.e., alarms raised \emph{before} the changepoint)
being bounded (at least with a high probability) by a specified constant.
Another desideratum for the conformal CUSUM procedure
is its \emph{efficiency}:
we would like an alarm to be raised soon \emph{after} the changepoint.
We did not investigate the efficiency of our change detection procedures
in \cite{Vovk/etal:2022book} since the task of proving efficiency
is in a sense too ambitious \cite[Sect.~8.6.3]{Vovk/etal:2022book}:
conformal testing may be based on powerful algorithms, such as deep neural networks,
involving elements of intelligence,
and there is little hope of guaranteeing their efficiency.

In this paper we investigate the efficiency of conformal change detection
in a very roundabout way,
which is an instance of what we called the Burnaev--Wasserman programme
in \cite[Sect.~2.5]{Vovk/etal:2022book}.
Suppose we model our observations as coming in the IID fashion from a distribution $Q_0$
before the changepoint and from another distribution $Q_1$ after the changepoint.
This is, however, a ``soft model'':
we do not want the validity of our procedure to depend on it.
The soft model provides us with a means of establishing the efficiency of our procedure:
the efficiency is measured under this model.

We started implementing this programme in the easier case of conformal e-testing
in \cite{Vovk/etal:arXiv2006v2}.
In this paper we extend it to conformal testing,
which has important advantages over conformal e-testing
\cite[Sect.~7]{Vovk/etal:arXiv2006v2}.

We start the main part of this paper from a description of the conformal CUSUM procedure
in Sect.~\ref{sec:CUSUM}.
The conformal CUSUM procedure is determined by a threshold and an underlying conformal test martingale,
and the latter is in turn determined by its underlying nonconformity measure and betting martingale.
We define asymptotically optimal, in a natural sense,
nonconformity measure and betting martingale in Sect.~\ref{sec:optimal}.
The betting martingale consists in applying the same betting function at each step,
and the betting function is explicitly computed in two popular special cases
(both involving the Gaussian model)
in Sect.~\ref{sec:special}.
An experimental section, Sect.~\ref{sec:experiments},
illustrates the efficiency of the asymptotically optimal conformal CUSUM procedure
that we derive in the previous sections,
and Sect.~\ref{sec:theory} is an attempt of a theoretical analysis.
Finally, we summarize and list some directions of further research in Sect.~\ref{sec:conclusion}.

In this paper we will refer to the assumption that the observations are IID
as the assumption of randomness.
By de Finetti's theorem this assumption is equivalent to that of exchangeability
for observations taking values in a standard Borel space.

\section{Conformal and standard CUSUM procedures}
\label{sec:CUSUM}

Let $S$ be a conformal test martingale (CTM) as defined in \cite[Sect.~8.1.2]{Vovk/etal:2022book};
for simplicity we will assume that it is positive, $S>0$.
For a given threshold $c>1$, the \emph{conformal CUSUM procedure}
raises the $k$th alarm, $k=0,1,\dots$, at the time
\begin{equation}\label{eq:CUSUM}
  \tau_k
  :=
  \min
  \left\{
    n>\tau_{k-1}:
    \max_{i=\tau_{k-1}\dots,n-1}
    \frac{S_n}{S_i}
    \ge
    c
  \right\},
\end{equation}
where $\tau_0:=0$ and $\min\emptyset:=\infty$.
A universally applicable but conservative property of validity of the conformal CUSUM procedure
is that
\begin{equation}\label{eq:validity-1}
  \E(\tau_k-\tau_{k-1})\ge c
\end{equation}
for IID observations.
This implies (albeit not trivially) that, under randomness,
\begin{equation}\label{eq:validity-2}
  \limsup_{n\to\infty}
  \frac{A_n}{n}
  \le
  \frac{1}{c}
  \quad
  \text{a.s.},
\end{equation}
where $A_n:=\max\{k:\tau_k\le n\}$ is the number of alarms raised by the conformal CUSUM procedure
by time $n$.

Next, with a view towards the Burnaev--Wasserman programme,
we consider the standard CUSUM procedure.
Let us use the notation of \cite{Vovk/etal:arXiv2006v2}:
the observation space is $\mathbf{Z}$ (a measurable space),
the pre-change distribution is $Q_0$,
and the post-change distribution is $Q_1$.
Let $f_0$ be the probability density of the pre-change distribution $Q_0$
and $f_1$ be the probability density of the post-change distribution $Q_1$
with respect to some $\sigma$-finite measure $m$ (such as $m:=Q_0+Q_1$).
The \emph{likelihood ratio} is
\begin{equation}\label{eq:LR}
  L(z)
  :=
  \frac{f_1(z)}{f_0(z)},
  \quad
  z\in\mathbf{Z};
\end{equation}
for simplicity we will assume that both $f_0$ and $f_1$ are positive, so that $L\in(0,\infty)$.
The corresponding \emph{likelihood ratio martingale} (LRM) is
\begin{equation}\label{eq:LRM}
  S_n
  :=
  \prod_{i=1}^n
  L(z_i),
  \quad
  n=0,1,\dots.
\end{equation}
The \emph{standard CUSUM procedure} corresponding to a threshold $c>1$
and the pre-/post-change distributions $(Q_0,Q_1)$
is defined by \eqref{eq:CUSUM},
where $S$ is the LRM \eqref{eq:LRM}.

The standard CUSUM procedure satisfies the properties of validity,
\eqref{eq:validity-1} and \eqref{eq:validity-2},
that we mentioned earlier for the conformal CUSUM procedure.
As we also mentioned earlier,
these properties of validity are very conservative for CUSUM
(since, as shown in, e.g., \cite[Chap.~4]{Vovk/etal:2022book},
they are also applicable to the Shiryaev--Roberts procedure, which raises many more alarms).
A possible solution is to use computer simulations
to find a suitable value of the threshold $c$
for a given target value for the frequency of false alarms
(as in \cite[Sect.~8.4.3]{Vovk/etal:2022book},
where we also find confidence intervals for the frequency of false alarms).
Many important pairs $(Q_0,Q_1)$ of pre-/post-change distributions are even amenable to theoretic analysis,
and it is possible for them to derive very accurate approximations to the value of $c$
leading to a chosen ``ARL2FA'' (``average run length to false alarm''),
a standard measure of validity for the standard CUSUM procedure.
In particular,
analysis of a discrete-time Gaussian case is performed in \cite{Goel/Wu:1971}
(and described in \cite[Sect.~5.2.2.1]{Basseville/Nikiforov:1993};
see also \cite[Sect.~8.2.6.6]{Tartakovsky/etal:2015}).

We would like to construct a CTM
approaching in its validity and efficiency the LRM
when both are plugged into the CUSUM scheme \eqref{eq:CUSUM}.
In the next section we will see that in the case of validity
this goal can achieved perfectly
(see Proposition~\ref{prop:validity}).
In the case of efficiency, our results are much weaker
(Sections~\ref{sec:experiments} and~\ref{sec:theory}).

\section{Asymptotically optimal CTM}
\label{sec:optimal}

First we briefly give a general definition,
which we then specialize to the more intuitive case of a finite observation space $\mathbf{Z}$.
This specialization will allow us to justify our definition in a simple setting.

An \emph{asymptotically optimal betting function} $f:[0,1]\to[0,\infty)$
for the pair $(Q_0,Q_1)$ of the pre-/post-change distributions
is defined by the condition
\[
  Q_0(\{z:L(z)>f(p)\})
  \le
  p
  \le
  Q_0(\{z:L(z)\ge f(p)\}),
\]
required to hold for any $p\in[0,1]$.
In other words, it is defined to be an inverse function of the \emph{survival function}
$\bar F(t):=Q_0(\{z:L(z)>t\})$ of the likelihood ratio $L$ under $Q_0$
\cite[Definition A.18]{Follmer/Schied:2016}.
This function is always decreasing (not necessarily strictly decreasing).

The \emph{canonical asymptotically optimal} (CAO) betting function $f$
is defined to be the largest function in this class, i.e.,
\[
  f(p)
  :=
  \inf\{t:\bar F(t)<p\}
  =
  \sup\{t:\bar F(t)\ge p\}.
\]
In other words, the CAO betting function is the left-continuous inverse function
of the survival function $\bar F$
(adapt \cite[Lemma A.19]{Follmer/Schied:2016} to the case of a decreasing function).
Equivalently,
the CAO betting function can be defined by $f(p)=F^{-1}(1-p)$,
where $F^{-1}$ is the upper quantile function
of the cumulative distribution function $F$ of $L$ under $Q_0$.

Let us now assume that the observation space $\mathbf{Z}$ is finite;
we will drop the curly braces in the notation such as $Q_0(\{z\})$ for $z\in\mathbf{Z}$.
To specialize the general definition to the case of a finite $\mathbf{Z}$,
sort the likelihood ratios
\[
  L(z)
  =
  Q_1(z)/Q_0(z)
\]
(cf.\ \eqref{eq:LR}) in the ascending order.
Let the resulting unique values for $L(z)$ be $l_1<\dots<l_K$. 
Set $q_k:=Q_0(\{z:L(z)=l_{k}\})$ for $k=1,\dots,K$.
An asymptotically optimal betting function $f:[0,1]\to[0,\infty)$
is defined by $f(p):=l_k$ for any $p\in[0,1]$,
where $k$ is defined by the condition
\begin{equation}\label{eq:betting}
  q_{k+1}+\dots+q_{K}
  \le
  p
  \le
  q_k\dots+q_{K}
\end{equation}
(this condition does not define $k$ uniquely if $p$ has the form $q_j+\dots+q_K$
for some $j\in\{2,\dots,K\}$,
but the probability of a uniformly distributed $p\in[0,1]$ having this form is zero).

Let us check that the term ``asymptotically optimal'' is indeed justified.
Let $N_0$ be the number of pre-change observations.
Consider a very large $N_0$ (so that we are talking about the asymptotics $N_0\to\infty$)
and a fixed $n$ (or $n$ bounded above by a fixed constant).
Take the likelihood ratio $L(z)$ as the nonconformity score of an observation $z$.
The conformal p-value at the $(N_0+n)$th step is
\begin{equation}\label{eq:conformal-p-value}
  p_{N_0+n}
  =
  \frac{
    \left|\left\{i:L_i>L_{N_0+n}\right\}\right|
    +
    \tau
    \left|\left\{i:L_i=L_{N_0+n}\right\}\right|
  }{N_0+n}
  \approx
  q_{k+1}+\dots+q_K+\tau q_k,
\end{equation}
where $i$ ranges over $\{1,\dots,N_0+n\}$, $\tau\in[0,1]$ is a random number,
and $k$ is defined by the condition $l_k=L_{N_0+n}$.
Since $n$ is fixed and $N_0\to\infty$, we can assume that $i$ ranges over $\{1,\dots,N_0\}$
and replace $N_0+n$ by $N_0$ in the denominator of \eqref{eq:conformal-p-value}.
The approximate equality ``$\approx$'' in \eqref{eq:conformal-p-value}
then follows from the law of large numbers (with a large probability).
The asymptotically optimal betting function takes value $l_k$
at this p-value by its definition (see \eqref{eq:betting}),
as it should (if we regards the LRM as being optimal).

An \emph{asymptotically optimal CTM} is a CTM
based on an asymptotically optimal betting function $f$
and on likelihood ratio $L$ as nonconformity measure.
It is the \emph{CAO CTM} if $f$ is the CAO betting function.
The following proposition says that the conformal CUSUM procedure
based on the CAO CTM, or any other asymptotically optimal CTM, is perfectly valid.

\begin{proposition}\label{prop:validity}
  For any pair $(Q_0,Q_1)$ of pre-/post-change distributions,
  the distribution of the likelihood ratio martingale under $Q_0^{\infty}$
  coincides with the distribution of any asymptotically optimal CTM
  under randomness (not necessarily under $Q_0^{\infty}$).
\end{proposition}

\begin{proof}
  Let $S$ be the LRM and $\bar S$ be an asymptotically optimal CTM
  based on $(Q_0,Q_1)$.
  The relative increments $S_n/S_{n-1}$, $n=1,2,\dots$,
  are independent because the observations are independent,
  and the relative increments $\bar S_n/\bar S_{n-1}$, $n=1,2,\dots$,
  are independent because the conformal p-values are independent
  \cite[Theorem~11.1]{Vovk/etal:2022book}.
  Therefore, it suffices to prove that, for a fixed $n$,
  $S_n/S_{n-1}$ and $\bar S_n/\bar S_{n-1}$ have the same distribution
  (the former under $Q_0^{\infty}$ and the latter under randomness).
  The former's distribution is the pushforward of $Q_0$
  by the likelihood ratio mapping \eqref{eq:LR}.
  The latter's distribution is identical
  by definition and \cite[Lemma~A.23]{Follmer/Schied:2016}.
\end{proof}

According to Proposition~\ref{prop:validity},
the conformal CUSUM procedure satisfies the same properties of validity
as the standard CUSUM procedure but under radically weaker assumptions:
instead of assuming $Q^{\infty}$, we just assume randomness.
In particular, we can use the whole arsenal of the existing results
in the standard theory of change detection to find a suitable threshold $c$
for the CUSUM procedure.
(See the mini-review at the end of Sect.~\ref{sec:CUSUM}.)
Alternatively, we can use computer simulations, as in \cite[Sect.~8.4.3]{Vovk/etal:2022book}.

Proposition~\ref{prop:validity} demonstrates a significant advantage
of the conformal CUSUM procedure over the conformal CUSUM e-procedure;
it is unlikely that anything similar to Proposition~\ref{prop:validity}
can be proved for the conformal CUSUM e-procedure,
since we know the distribution of the conformal p-values (it is uniform in $[0,1]^{\infty}$)
but not of conformal e-values.

\section{Computing the CAO betting function in popular special cases}
\label{sec:special}

Now we take the observation space to be the real line $\R$, $\mathbf{Z}=\R$,
and take $m$ to be Lebesgue measure;
we still assume that the pre-change and post-change distributions
are absolutely continuous w.r.\ to $m$ with positive probability densities.
As before, $f_0$ is the probability density of the pre-change distribution $Q_0$
and $f_1$ is the probability density for the post-change distribution $Q_1$.
The likelihood ratio is defined by \eqref{eq:LR}, where now $z\in\R$.

The special cases for which we will compute the CAO betting function
in this section will all involve the Gaussian statistical model.
First let us assume that the pre-change distribution is $Q_0:=N(0,1)$
and the post-change distribution is $Q_1:=N(\mu,1)$.
For concreteness, let us assume $\mu>0$
(this does not really restrict generality because of symmetry).
The likelihood ratio is
\[
  L(z)
  =
  \exp(\mu z-\mu^2/2).
\]
Solving the inequality $L(z)\ge l$ for $l\in(0,\infty)$,
we obtain the optimal Neyman--Pearson region
\begin{equation}\label{eq:mean-region}
  z
  \ge
  \frac{\ln l}{\mu} + \frac{\mu}{2}.
\end{equation}
Its probability under $Q_0$
(i.e., the \emph{type I error probability} of the corresponding statistical test)
is
\begin{equation*}
  p
  =
  \P
  \left(
    Z
    \ge
    \frac{\ln l}{\mu} + \frac{\mu}{2}
  \right)
  =
  1
  -
  \Phi
  \left(
    \frac{\ln l}{\mu} + \frac{\mu}{2}
  \right),
\end{equation*}
where $Z\sim N(0,1)$ and $\Phi$ stands for the distribution function of $N(0,1)$.
Expressing $l$ in terms of $p$,
we obtain the betting function
\begin{equation}\label{eq:mean-opt-p-s}
  f(p)
  =
  l
  =
  \exp
  \left(
    \mu\Phi^{-1}(1-p)
    -
    \mu^2/2
  \right),
\end{equation}
$\Phi^{-1}$ being the inverse function of $\Phi$
(i.e., $\Phi^{-1}$ is the quantile function of $N(0,1)$).
If $\mu<0$, we should replace the ``$\ge$'' in \eqref{eq:mean-region} by ``$\le$'',
which leads to
\begin{equation}\label{eq:mean-opt-n-s}
  f(p)
  =
  l
  =
  \exp
  \left(
    \mu \Phi^{-1}(p)
    -
    \mu^2/2
  \right)
\end{equation}
in place of \eqref{eq:mean-opt-p-s}.
(We can also see that \eqref{eq:mean-opt-p-s} implies \eqref{eq:mean-opt-n-s}
by the symmetry between $p$ and $1-p$.)

Next we consider the case where the pre-change distribution is $N(0,1)$
and the post-change distribution is $N(0,\sigma^2)$.
The likelihood ratio is
\[
  L(z)
  =
  \frac{1}{\sigma}
  \exp
  \left(
    (1-1/\sigma^2) z^2/2
  \right).
\]
Let us first assume $\sigma>1$ (the observations are becoming
more spread out after the changepoint).
Then $L(z)$ ranges over $[1/\sigma,\infty)$.
Solving $L(z)\ge l$ for $l\ge1/\sigma$ gives the optimal Neyman--Pearson region
\begin{equation}\label{eq:sigma-region}
  \left(
    -\infty,
    -\sqrt{\frac{2\ln(l\sigma)}{1-1/\sigma^2}}
  \right]
  \cup
  \left[
    \sqrt{\frac{2\ln(l\sigma)}{1-1/\sigma^2}},
    \infty
  \right).
\end{equation}
The type I error probability is
\begin{multline}\label{eq:p_0-sigma}
  p
  =
  \P
  \left(
    Z
    \in
    \left(
      -\infty,
      -\sqrt{\frac{2\ln(l\sigma)}{1-1/\sigma^2}}
    \right]
    \cup
    \left[
      \sqrt{\frac{2\ln(l\sigma)}{1-1/\sigma^2}},
      \infty
    \right)
  \right)\\
  =
  2N
  \left(
    -\sqrt{\frac{2\ln(l\sigma)}{1-1/\sigma^2}}
  \right).
\end{multline}
For the change $N(0,1)\to N(0,\sigma^2)$ for $\sigma>1$,
\eqref{eq:p_0-sigma} gives
\begin{equation}\label{eq:sigma-opt-p-s}
  f(p)
  =
  l
  =
  \frac{1}{\sigma}
  \exp
  \left(
    \frac{1-1/\sigma^2}{2}
    \left(
      \Phi^{-1}(p/2)
    \right)^2
  \right)
  \in
  \left[
    1/\sigma,\infty
  \right).
\end{equation}
Similarly, we get
\begin{equation}\label{eq:sigma-opt-n-s}
  f(p)
  =
  l
  =
  \frac{1}{\sigma}
  \exp
  \left(
    \frac{1-1/\sigma^2}{2}
    \left(
      \Phi^{-1}((1-p)/2)
    \right)^2
  \right)
  \in
  \left(
    0,1/\sigma
  \right]
\end{equation}
for $\sigma<1$;
this is obtained by replacing the optimal Neyman-Pearson region~\eqref{eq:sigma-region}
by its complement.

There is a convenient equivalent definition of the CAO betting function
that sometimes leads to nicer and more intuitive mathematical expressions.
For each threshold $l\in(0,\infty)$,
find, alongside the type I error probability
\begin{equation}\label{eq:p_0}
  p_0(l)
  :=
  \int_{z:L(z)\ge l}
  f_0(z)
  \d z
  \in[0,1],
\end{equation}
the probability
\begin{equation}\label{eq:p_1}
  p_1(l)
  :=
  \int_{z:L(z)\ge l}
  f_1(z)
  \d z
  \in[0,1]
\end{equation}
(this is the \emph{power}, or 1 minus the \emph{type II error probability}).
Let us assume, for simplicity,
that the functions $p_0$ and $p_1$ defined by \eqref{eq:p_0} and \eqref{eq:p_1}
are continuous and strictly increasing over some interval.
(These conditions are satisfied in the examples considered in this section.)
The \emph{Neyman--Pearson ROC curve} $R$ for the pair $(Q_0,Q_1)$
(with $Q_0$ interpreted as null hypothesis and $Q_1$ as alternative)
is defined to be $p_1$ as function of $p_0$:
\[
  R(p)
  :=
  \sup
  \{
    p_1(l):
    p_0(l)=p
  \}.
\]
This definition is motivated, of course,
by the Neyman--Pearson lemma \cite[Sect.~3.2]{Lehmann/Romano:2022}.

\begin{proposition}\label{prop:opt}
  The Neyman--Pearson ROC curve $R$ is differentiable,
  and the CAO betting function $f$ is equal to its derivative,
  $f=R'$.
\end{proposition}

\noindent
Another way to express Proposition~\ref{prop:opt} is:
the CAO betting function $f$ is equal to the minus derivative
of the type II error probability $1-p_1$
expressed as function of the type I error probability $p_0$.

\begin{proof}[Proof of Proposition~\ref{prop:opt}]
  Under our assumptions, as $\Delta l\to0$,
  \begin{align*}
    \frac{\d p_1}{\d p_0}(l)
    &\sim
    \frac{p_1(l+\Delta l)-p_1(l)}{p_0(l+\Delta l)-p_0(l)}\\
    &=
      \int_{z:L(z)\in[l,l+\Delta l]}
      L(z) f_0(z)
      \dd z
      \Big/\!\!  % alternative: \middle/
      \int_{z:L(z)\in[l,l+\Delta l]}
      f_0(z)
      \dd z
    \sim l.
    \qedhere
  \end{align*}
\end{proof}

Proposition~\ref{prop:opt} gives more intuitive expressions
for the CAO betting functions discussed so far.
Start, again, from the change $N(0,1)\to N(\mu,1)$ for $\mu>0$.
The probability of \eqref{eq:mean-region} under $Q_1$ (i.e., the power) is
\[
  p^*
  =
  \P
  \left(
    Z+\mu
    \ge
    \frac{\ln l}{\mu} + \frac{\mu}{2}
  \right)
  =
  1
  -
  \Phi
  \left(
    \frac{\ln l}{\mu} - \frac{\mu}{2}
  \right).
\]
(So that now we are writing $p$ for $p_0$ and $p^*$ for $p_1$.)
Expressing $p^*$ via $p$, we obtain the Neyman--Pearson ROC curve
\[
  R(p)
  =
  p^*
  =
  1-\Phi(\Phi^{-1}(1-p)-\mu).
\]
Differentiating this function gives the optimal betting function
\begin{equation}\label{eq:mean-opt-p}
  f(p)
  =
  R'(p)
  =
  \frac
  {\phi(\Phi^{-1}(1-p)-\mu)}
  {\phi(\Phi^{-1}(1-p))},
\end{equation}
where $\phi$ is the probability density of $N(0,1)$.
It is easy to check that the ex\-pres\-sion~\eqref{eq:mean-opt-p}
is equivalent to our previous expression~\eqref{eq:mean-opt-p-s}.
If $\mu<0$, replacing the ``$\ge$'' in \eqref{eq:mean-region} by ``$\le$''
gives
\begin{equation}\label{eq:mean-opt-n}
  f(p)
  =
  R'(p)
  =
  \frac
  {\phi(\Phi^{-1}(p)-\mu)}
  {\phi(\Phi^{-1}(p))}
\end{equation}
in place of \eqref{eq:mean-opt-p};
the last expression is again equivalent to~\eqref{eq:mean-opt-n-s}.

We can interpret the betting functions \eqref{eq:mean-opt-p} and \eqref{eq:mean-opt-n}
as likelihood ratios.
For example, \eqref{eq:mean-opt-p} is the likelihood ratio
of the alternative hypothesis $Q_1$ to the null hypothesis $Q_0$
after transforming the observed p-value $p$ to the Gaussian scale $\Phi^{-1}(1-p)$.

In the case $N(0,1)\to N(0,\sigma^2)$ with $\sigma>1$,
the power is
\begin{multline*}
  p^*
  =
  \P
  \left(
    \sigma Z
    \in
    \left(
      -\infty,
      -\sqrt{\frac{2\ln(l\sigma)}{1-1/\sigma^2}}
    \right]
    \cup
    \left[
      \sqrt{\frac{2\ln(l\sigma)}{1-1/\sigma^2}},
      \infty
    \right)
  \right)\\
  =
  2N
  \left(
    -\frac{1}{\sigma}\sqrt{\frac{2\ln(l\sigma)}{1-1/\sigma^2}}
  \right).
\end{multline*}
Therefore, the Neyman--Pearson ROC curve is
\[
  R(p)
  =
  p^*
  =
  2N
  \left(
    \frac{1}{\sigma}
    \Phi^{-1}(p/2)
  \right),
\]
and differentiation gives the betting function
\begin{equation}\label{eq:sigma-opt-p}
  f(p)
  =
  R'(p)
  =
  \frac{\phi\left(\frac{1}{\sigma}\Phi^{-1}(p/2)\right)}{\sigma\phi(\Phi^{-1}(p/2))}.
\end{equation}
The formula for the betting function for the case $\sigma<1$ is
\begin{equation}\label{eq:sigma-opt-n}
  f(p)
  =
  R'(p)
  =
  \frac{\phi\left(\frac{1}{\sigma}\Phi^{-1}((1-p)/2)\right)}{\sigma\phi(\Phi^{-1}((1-p)/2))};
\end{equation}
as usual,
it is obtained by replacing~\eqref{eq:sigma-region} by its complement.
Of course, \eqref{eq:sigma-opt-p} and \eqref{eq:sigma-opt-n}
are equivalent to \eqref{eq:sigma-opt-p-s} and \eqref{eq:sigma-opt-n-s},
but the former look more natural.

\begin{remark}
  The conformal test martingales based on the CAO betting functions
  developed in this section
  can be integrated over a probability measure on the parameter
  (the post-change mean $\mu$ or the post-change standard deviation $\sigma$).
  This allows us to achieve good performance for wider ranges of post-change distributions.
\end{remark}

\begin{remark}
  We can't easily combine the two Gaussian cases that we have considered,
  $N(0,1)\to N(\mu,1)$ and $N(0,1)\to N(0,\sigma^2)$,
  into the case of the pre-change distribution being $N(0,1)$
  and the post-change distribution being $N(\mu,\sigma^2)$.
  The general case of a change in both location and scale parameters
  does not appear to admit a simple closed-form solution.
\end{remark}

\section{Efficiency and validity: Experimental results}
\label{sec:experiments}

In this section we will report some experimental results for conformal CUSUM procedures
and, for comparison, conformal CUSUM e-procedures considered in \cite{Vovk/etal:arXiv2006v2}.
The latter are based on universal e-values
\begin{equation}\label{eq:E-from-L}
  E_n
  :=
  \frac{L_n}{\frac1n(L_1+\dots+L_n)},
\end{equation}
where $L_n:=L(z_n)$,
and no further tuning to $Q_0$ and $Q_1$
(like choosing the betting functions in the case of conformal testing)
is needed.
We omit the details referring the reader to \cite{Vovk/etal:arXiv2006v2}.

\begin{figure}[bt]
  \begin{center}
    \includegraphics[width=0.48\textwidth]{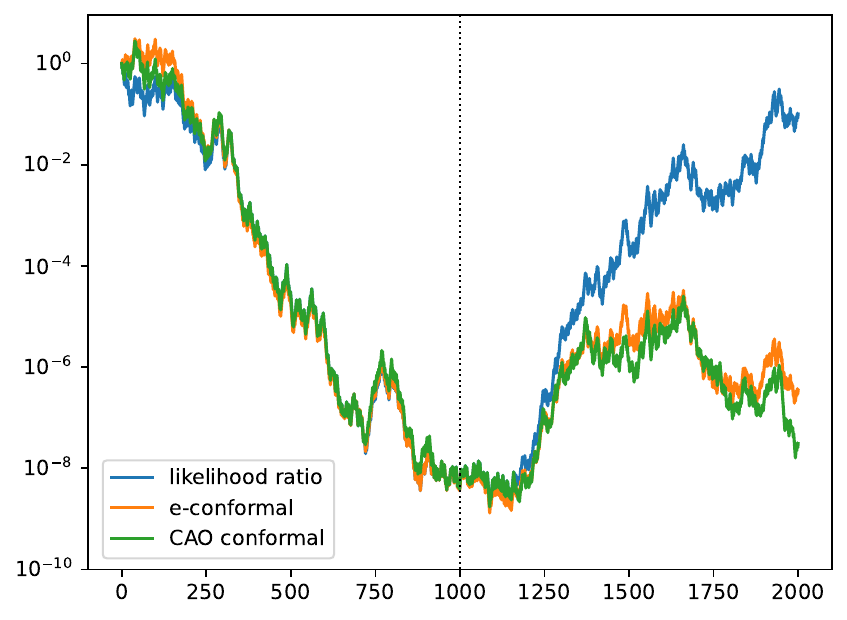}
    \includegraphics[width=0.48\textwidth]{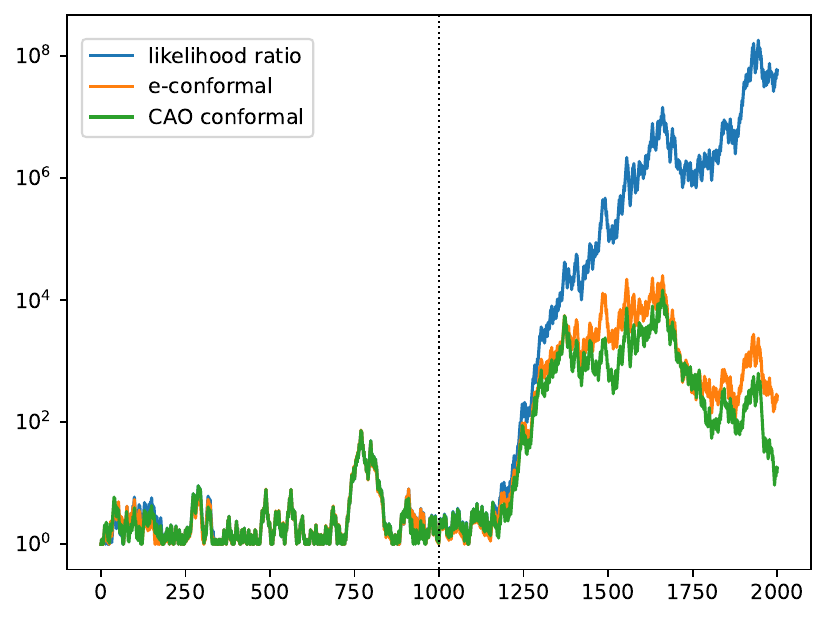}
  \end{center}
  \caption{Left panel: the likelihood ratio martingale, conformal e\-/pseudomartingale,
    and conformal test martingale, as described in text, in the Bernoulli case
    (with the parameter equal to $0.5$ pre-change and $0.6$ post-change).
    Right panel: the corresponding CUSUM statistics.}
  \label{fig:Bernoulli}
\end{figure}

Experimental results for Bernoulli observations are given in Fig.~\ref{fig:Bernoulli}.
The first $N_0:=1000$ observations $z_1,\dots,z_{N_0}$ are generated
from the Bernoulli distribution with parameter $0.5$
and another $N_1:=1000$ observations $z_{N_0+1},\dots,z_{N_0+N_1}$
from the Bernoulli distribution with parameter $0.6$.
The changepoint 1000 is shown as a dashed vertical line.
The left panel of Fig.~\ref{fig:Bernoulli} shows three paths:
\begin{itemize}
\item
  the likelihood ratio martingale $L_1\dots L_n$, $n=0,1,\dots,N_0+N_1$, in blue;
\item
  the conformal e-pseudomartingale $E_1\dots E_n$, $n=0,1,\dots,N_0+N_1$, in orange,
  with the conformal e-values defined as the normalized $L_n$:
  see \eqref{eq:E-from-L};
\item
  the CAO conformal test martingale in green;
  as the green line was drawn after the orange line and the two lines are very close,
  the green line often obscures the orange one.
\end{itemize}
The right panel shows the corresponding \emph{CUSUM statistics} $S^*_n$ of these processes
defined as
\[
  S^*_n
  :=
  S_n/\min(S_0,\dots,S_{n-1}),
\]
$S_n$ being the original process.
The CAO betting function is,
according to the definition above,
\[
  f(p)
  :=
  \begin{cases}
    1.2 & \text{if $p\le0.5$}\\
    0.8 & \text{otherwise}.
  \end{cases}
\]

In Fig.~\ref{fig:Bernoulli} we can clearly see a phenomenon that we called \emph{decay}
in \cite[Sect.~8.4]{Vovk/etal:2022book}:
when testing the assumption of randomness (the orange and green lines),
the new distribution $Q_1$ will gradually become the ``new normal'' after the changepoint,
and the growth of the CTM will slow down or stop.
Of course, there is no decay for the blue line.
Decay is inevitable when testing randomness,
and by itself it does not indicate any inefficiency.

\begin{figure}[bt]
  \begin{center}
    \includegraphics[width=0.48\textwidth]{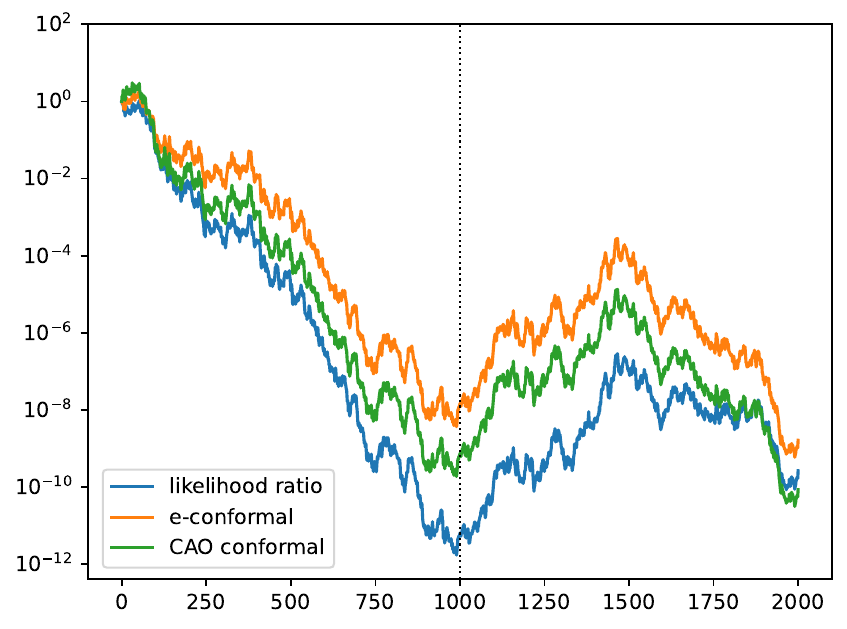}
    \includegraphics[width=0.48\textwidth]{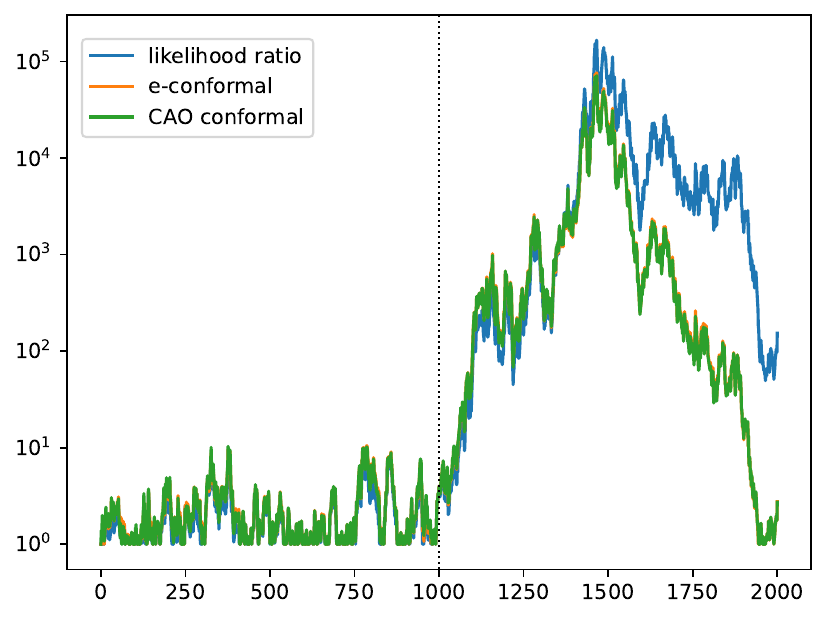}
  \end{center}
  \caption{Left panel: the three processes as in Fig.~\ref{fig:Bernoulli},
    but with a change from $N(0,1)$ to $N(0.2,1)$.
    Right panel: the corresponding CUSUM statistics.}
  \label{fig:Gauss-mean}
\end{figure}

Figure~\ref{fig:Gauss-mean} is the analogue of Fig.~\ref{fig:Bernoulli}
for a change in mean for a Gaussian distribution.
Now the first $N_0:=1000$ observations are generated from $N(0,1)$
and the other $N_1:=1000$ from $N(0.2,1)$.
The orange and green lines are remarkably close to each other in the right panel
(and this remains true even if we vary the seed of the random number generator;
for other seeds, of course, the blue line does not show any signs of decay,
which is an artefact of our default choice of ``42'' as seed).

\begin{figure}[bt]
  \begin{center}
    \includegraphics[width=0.48\textwidth]{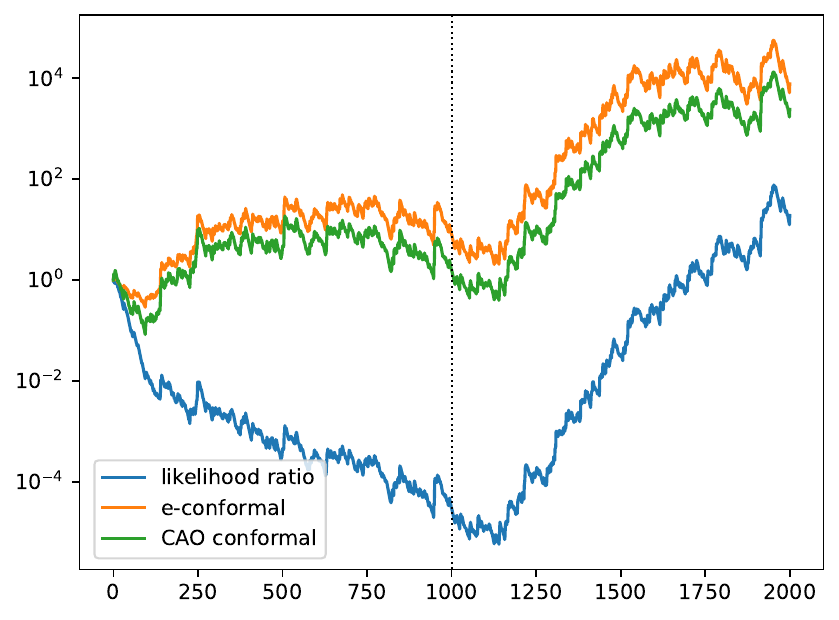}
    \includegraphics[width=0.48\textwidth]{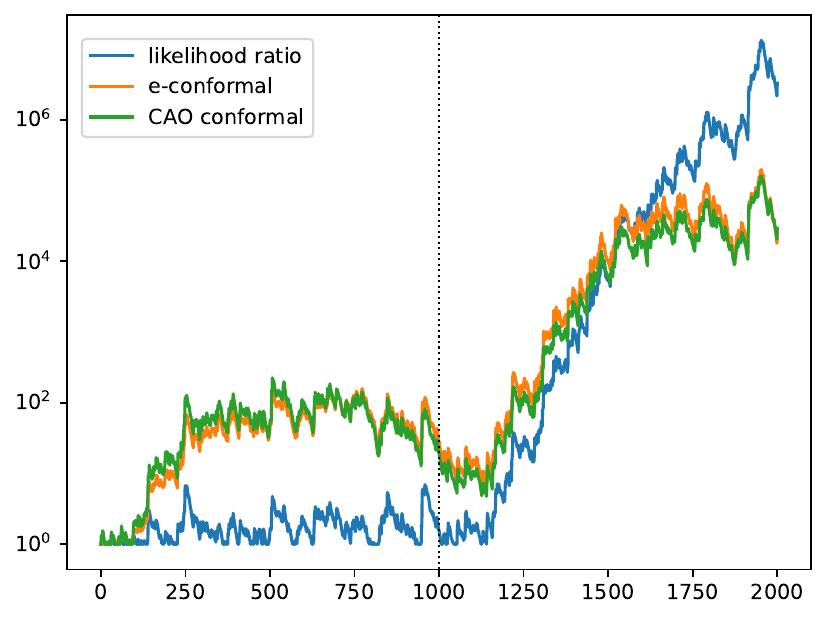}
  \end{center}
  \caption{Left panel: the three processes as in Fig.~\ref{fig:Bernoulli},
    but with a change from $N(0,1)$ to $N(0,1.1)$.
    Right panel: the corresponding CUSUM statistics.}
  \label{fig:Gauss-sigma-p}
\end{figure}

\begin{figure}[bt]
  \begin{center}
    \includegraphics[width=0.48\textwidth]{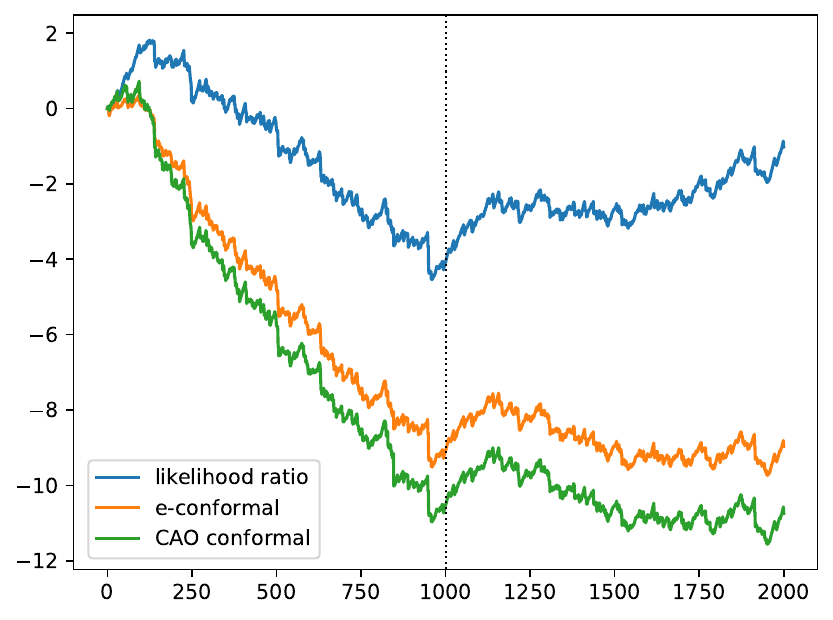}
    \includegraphics[width=0.48\textwidth]{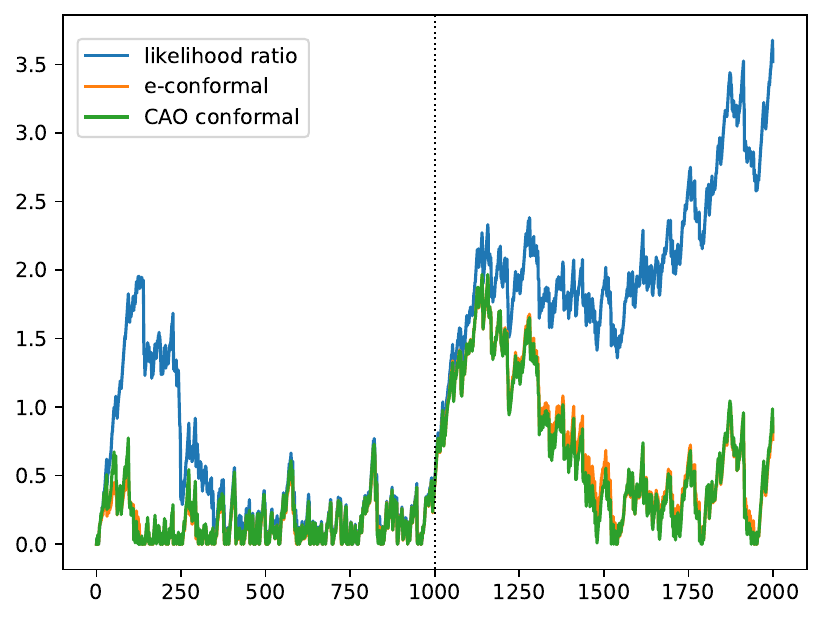}
  \end{center}
  \caption{Left panel: the three processes as in Fig.~\ref{fig:Bernoulli},
    but with a change from $N(0,1)$ to $N(0,0.9)$.
    Right panel: the corresponding CUSUM statistics.}
  \label{fig:Gauss-sigma-n}
\end{figure}

Figure~\ref{fig:Gauss-sigma-p} is the analogue of Figures~\ref{fig:Bernoulli} and~\ref{fig:Gauss-mean}
for a change from $N(0,1)$ to $N(0,1.1)$.
The orange and green lines are still fairly close in the right panel.
And finally, Fig.~\ref{fig:Gauss-sigma-n} is the analogue for a change from $N(0,1)$ to $N(0,0.9)$.
We use the expressions \eqref{eq:mean-opt-p-s}, \eqref{eq:sigma-opt-p-s}, and \eqref{eq:sigma-opt-n-s}
for the CAO betting functions
(but of course using \eqref{eq:mean-opt-p}, \eqref{eq:sigma-opt-p}, and \eqref{eq:sigma-opt-n}
would have led to the identical results).

We will finish this section by exploring the validity of the conformal CUSUM procedure.
It might appear that there is a contradiction between Proposition~\ref{prop:validity}
and the left panels of Figures~\ref{fig:Gauss-sigma-p} and~\ref{fig:Gauss-sigma-n}:
those panels seem to suggest that the stochastic behaviour of the LRM and the CAO CTM
is very different before the changepoint.
In the rest of this section we will see that,
despite the frequent remarkably different behaviour of the LRM and the CAO CTM
before the changepoint on the same dataset,
their distributions are very close to each other
(and we know from Proposition~\ref{prop:validity} that they are even identical),
and both are typically close to the distribution of the conformal e-pseudomartingale
(for which we do not have a theoretical explanation at this time).

\begin{figure}[bt]
  \begin{center}
    \includegraphics[width=0.48\textwidth]{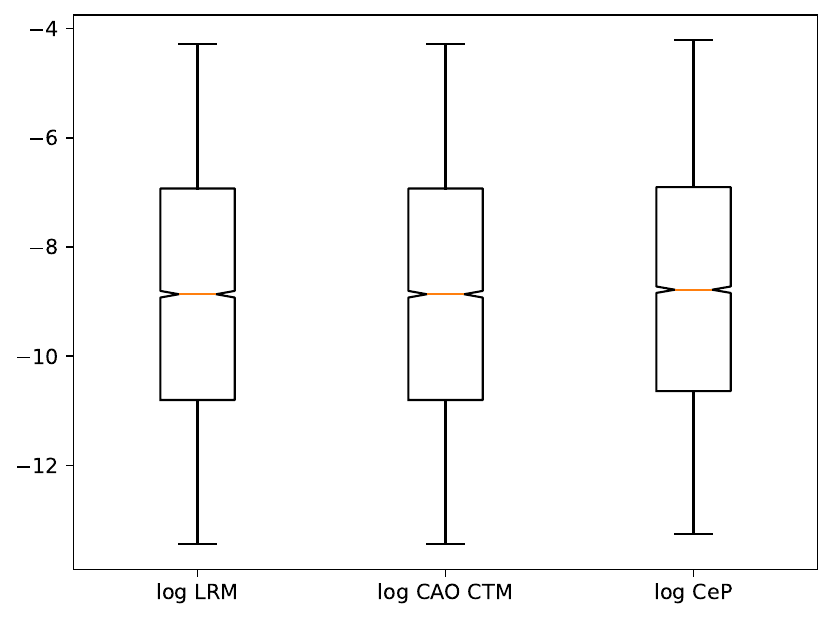}
    \includegraphics[width=0.48\textwidth]{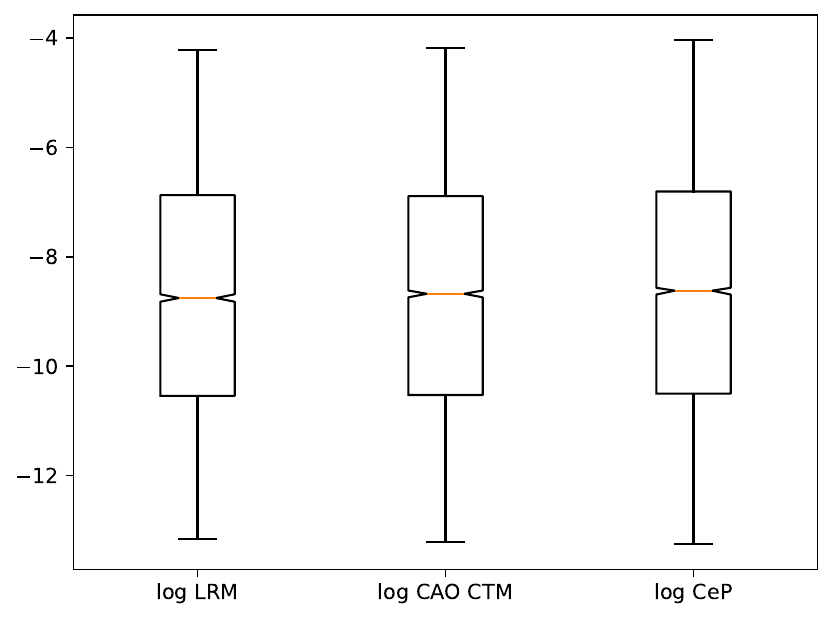}
  \end{center}
  \caption{Left panel: the validity for the Bernoulli case.
    Right panel: the validity for the Gaussian case with a change in mean.
    As described in text, with 10,000 simulations.}
  \label{fig:validity-0-1}
\end{figure}

Figure~\ref{fig:validity-0-1} demonstrates the closeness in two cases shown in its two panels.
We perform 10,000 simulations of 1000 observations from a pre-change distribution
computing the final values of three processes.
The boxplots in the figure show the median
(as the orange middle line with notches showing its default Matplotlib confidence limits),
the first and third quartiles (as the lower and upper sides of the boxes),
and the 5\% and 95\% quantiles (as the whiskers)
of the decimal logarithms of the three processes.
The left panel is for the Bernoulli case with the same parameters as before
($0.5$ pre-change and $0.6$ post-change),
and the right panel is for the Gaussian case with a change in mean with the same parameters
($N(0,1)$ pre-change and $N(0.2,1)$ post-change).
There is no changepoint, and we only generate $N_0:=1000$ observations
from the pre-change distributions.
The boxplots in each panel are given in this order:
the LRM, the CAO CTM, and the conformal e-pseudomartingale (CeP).
All three boxplots are very close in each of the two panels.

\begin{figure}[bt]
  \begin{center}
    \includegraphics[width=0.48\textwidth]{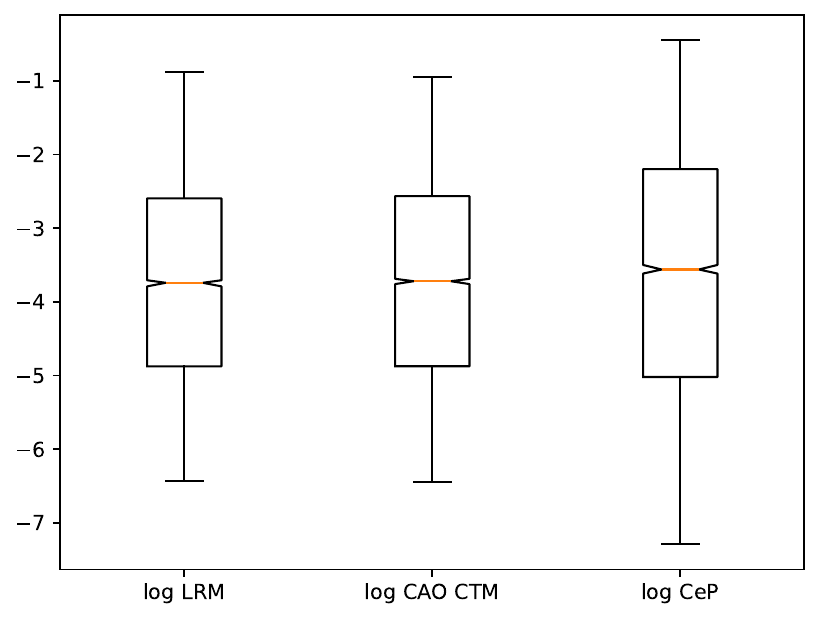}
    \includegraphics[width=0.48\textwidth]{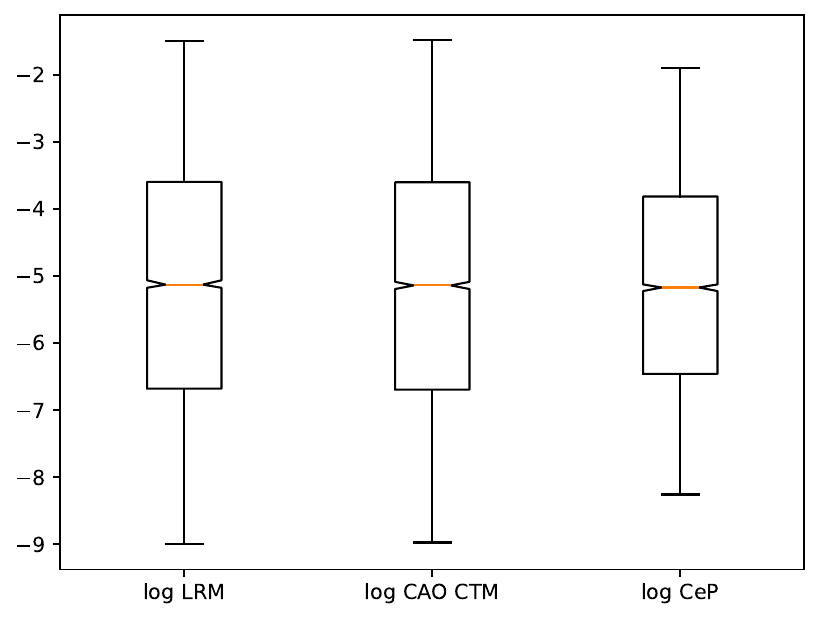}
  \end{center}
  \caption{Left panel: the validity for the Gaussian case with an upward change in variance.
    Right panel: the validity for the Gaussian case with a downward change in variance.
    The number of simulations is 10,000 in both panels; see the description in text.}
  \label{fig:validity-2p-2n}
\end{figure}

\begin{figure}[bt]
  \begin{center}
    \includegraphics[width=0.6\textwidth]{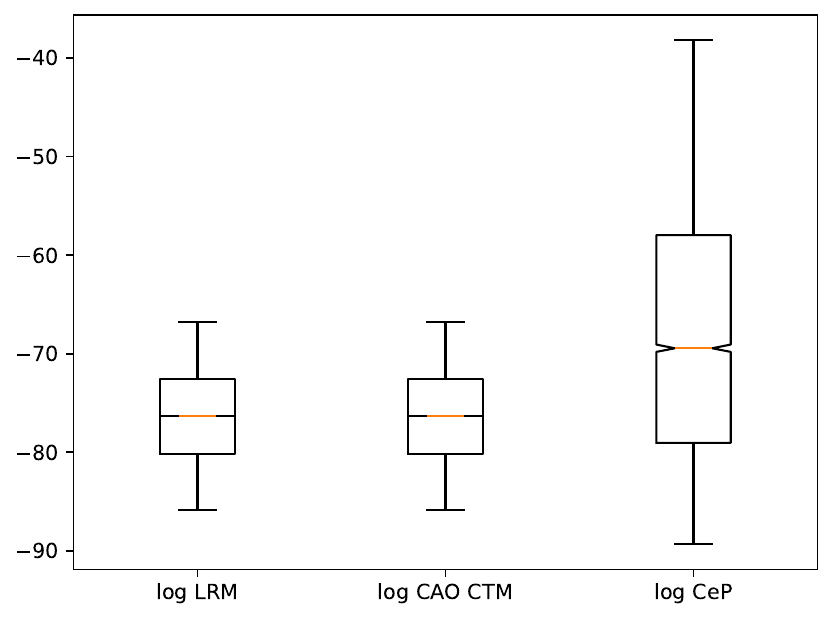}
  \end{center}
  \caption{Validity and lack thereof for a Bernoulli case, as described in text
    (with 10,000 simulations).}
  \label{fig:validity-broken}
\end{figure}

Figure~\ref{fig:validity-2p-2n} is the analogue of Fig.~\ref{fig:validity-0-1}
for the upward and downward change in the variance of the standard Gaussian distribution
generating the data, again with the same parameters as before.
While the first two boxplots are still very close to each other
(as they must be by Proposition~\ref{prop:validity}),
the boxplot for conformal e-pseudomartingale looks visibly different.
Figure~\ref{fig:validity-broken} gives a more extreme example of broken validity.
In it, the pre-change distribution is Bernoulli with parameter 0.1
while the post-change distribution is Bernoulli with parameter 0.9;
the number of observations is $N_0:=100$.

\section{Efficiency: Theoretical analysis}
\label{sec:theory}

In this section we report a simple preliminary theoretical result
about the efficiency of conformal testing in the context of change detection.
We only consider the Bernoulli case.

\begin{theorem}\label{thm:LR-vs-CT}
  Consider the Bernoulli case
  with the pre-change parameter $\theta_0$ and post-change parameter $\theta_1$.
  Let $B$ be given by
  \[
    B
    :=
    \left|
      \ln
      \frac{\theta_1(1-\theta_0)}{\theta_0(1-\theta_1)}
    \right|
  \]
  (this is a kind of distance between $\theta_0$ and $\theta_1$).
  Let $p_n$ be the $n$th conformal p-value
  based on the likelihood ratio nonconformity measure
  and $f$ be the CAO betting function.
  Then, for any $\epsilon>0$,
  the log ratio of the likelihood ratio martingale to the CAO CTM
  after the changepoint is bounded above as
  \begin{multline}\label{eq:LR-vs-CT}
    \forall n\in\{1,\dots,N_1\}:
    \ln
    \frac
    {L_{N_0+1}\dots L_{N_0+n}}
    {f(p_{N_0+1})\dots f(p_{N_0+n})}\\
    \le
    B
    \left(
      \ln\frac{2}{\epsilon}
      +
      5 N_1
      \sqrt{\frac{\ln\frac{4}{\epsilon}}{2N_0}}
      +
      \frac52
      \frac{N_1(N_1+1)}{N_0}
    \right)
  \end{multline}
  with probability at least $1-\epsilon$.
\end{theorem}

\noindent
Assuming that $\ln\frac{1}{\epsilon}$ is only moderately large,
the right-hand side of \eqref{eq:LR-vs-CT} will be moderately large
if $N_1$ has the same order of magnitude as $\sqrt{N_0}$
(this is similar to our conclusion about \cite[Proposition~8]{Vovk/etal:arXiv2006v2}).

For a proof of Theorem~\ref{thm:LR-vs-CT} see Appendix~\ref{app:proof}.

\section{Conclusion}
\label{sec:conclusion}

In this paper we started implementing the Burnaev--Wasserman programme
in the case of change detection.
We showed that the conformal CUSUM procedure has perfect validity
and investigated its efficiency both experimentally
and, in a very basic way, theoretically.

Our secondary goal was to compare conformal testing and conformal e-test\-/ing
in the context of change detection.
These are some advantages of conformal testing:
\begin{itemize}
\item
  The conformal CUSUM procedure is perfectly valid,
  as we saw in Proposition~\ref{prop:validity}.
  This allows us to find a suitable value of the parameter $c$
  for a given target value of the frequency of false alarms
  using computer simulations or existing theoretical results.
  Proposition~\ref{prop:validity} also demonstrates the perfect validity
  of the conformal version of the Shiryaev--Roberts procedure,
  which is an important competitor of CUSUM.
  Validity can be violated for the conformal CUSUM e-procedure,
  as we saw in Sect.~\ref{sec:experiments}
  (Fig.~\ref{fig:validity-broken}).
\item
  Conformal testing produces conformal test martingales,
  which are genuine martingales,
  and merging genuine martingales is easy:
  we can just average them.
  This does not work for conformal e-pseudomartingales,
  while it is important in change detection
  (e.g., when we do not have a single post-change distribution
  even as our soft model).
\end{itemize}
This short list complements the analogous list in \cite[Sect.~7]{Vovk/etal:arXiv2006v2}.
One item on that list is the need for randomization in conformal testing,
which may be regarded as its disadvantage.
Randomization is really harmful if it significantly affects the outcome of testing;
it becomes much more tolerable if (as is often the case) it is performed repeatedly
and the law of large numbers becomes applicable.

These are some open problems and directions of further research along these lines:
\begin{itemize}
\item
  We saw in Sect.~\ref{sec:experiments} that Proposition~\ref{prop:validity}
  does not hold for conformal e-testing
  (Fig.~\ref{fig:validity-broken})
  but in typical cases (Figures~\ref{fig:validity-0-1}--\ref{fig:validity-2p-2n})
  the distributions of the final values for likelihood ratio martingales
  and conformal e-pseudomartingales still look similar before the changepoint.
  Can we prove partial validity results for conformal e-testing?
\item
  The experimental results in Sect.~\ref{sec:experiments} and Theorem~\ref{thm:LR-vs-CT}
  show the closeness of the CUSUM statistics for the LRM and CAO CTM,
  but Theorem~\ref{thm:LR-vs-CT} covers only the Bernoulli case.
  It would be nice to have more general and stronger efficiency results.
\item
  Similarly, \cite{Vovk/etal:arXiv2006v2} demonstrated the closeness of the CUSUM statistics
  for the LRM and the conformal e-pseudomartingale \eqref{eq:E-from-L}.
  Those efficiency results can also be generalized and strengthened.
\item
  In Sect.~\ref{sec:experiments} we observed extreme closeness of the CUSUM statistics
  for conformal e-pseudomartingales and conformal test martingales.
  Is this a manifestation of a general efficiency result?
\end{itemize}

\appendix
\section{Proof of Theorem~\ref{thm:LR-vs-CT}}
\label{app:proof}

To compare conformal test martingales with likelihood ratio martingales,
we will use the martingale multiplicative version of the Chernoff bound
to be proved in Appendix~\ref{app:MMM-Chernoff} (Theorem~\ref{thm:MMM}).

The betting function in Theorem~\ref{thm:LR-vs-CT} is
\[
  f(p)
  :=
  \begin{cases}
    \frac{\theta_1}{\theta_0} & \text{if $p\le \theta_0$}\\
    \frac{1-\theta_1}{1-\theta_0} & \text{otherwise}.
  \end{cases}
\]
We split $\epsilon$ into two equal parts of $\epsilon/2$.
By Hoeffding's inequality (see, e.g., \cite[Sect.~A.6.3]{Vovk/etal:2022book}),
\[
  \P
  \left(
    \left|
      \frac{K_0}{N_0}
      -
      \theta_0
    \right|
    \ge
    \delta
  \right)
  \le
  2\exp(-2\delta^2N_0),
\]
where $K_0$ is the number of 1s in the first $N_0$ steps.
Solving the equation
\[
  2\exp(-2\delta^2N_0)
  =
  \frac{\epsilon}{2}
\]
gives
\begin{equation}\label{eq:delta}
  \delta
  =
  \sqrt{\frac{\ln\frac{4}{\epsilon}}{2N_0}}.
\end{equation}
Let us fix the first $N_0$ observations such that
\[
  \left|
    \frac{K_0}{N_0}
    -
    \theta_0
  \right|
  <
  \delta.
\]

We say that step $n$ after the changepoint is \emph{anomalous}
if one of the events $p_n\le\theta_0$ and $z_n=1$ happens
while the other does not happen.
Only if step $n$ after the changepoint is anomalous
can we have $L_{N_0+n}\ne f(p_{N_0+n})$,
and even if that step is anomalous we still have
\[
  \ln L_{N_0+n}-\ln f(p_{N_0+n})
  \le
  B.
\]
Therefore, our proof strategy will consist
in bounding the number of anomalous steps.

The expectation of the number of anomalous steps
up to (and including) step $N_1$ after the changepoint
is at most
\[
  \sum_{n=1}^{N_1}
  \left(
    \delta + \frac{n}{N_0}
  \right)
  =
  N_1 \delta + \frac{N_1(N_1+1)}{2N_0}.
\]
By the martingale multiplicative Chernoff bound,
\begin{multline}\label{eq:A}
  \P
  \left(
    A
    \ge
    (1+\Delta)
    \left(
      N_1 \delta + \frac{N_1(N_1+1)}{2N_0}
    \right)
  \right)\\
  \le
  \exp
  \left(
    -\frac{\Delta^2}{2+\Delta}
    \left(
      N_1 \delta + \frac{N_1(N_1+1)}{2N_0}
    \right)
  \right),
\end{multline}
where $A$ is the number of anomalous post-change steps and $\Delta$ is a positive constant.
Solving the equation
\[
  \exp
  \left(
    -\frac{\Delta^2}{2+\Delta}
    \left(
      N_1 \delta + \frac{N_1(N_1+1)}{2N_0}
    \right)
  \right)
  =
  \frac{\epsilon}{2}
\]
in $\Delta$ is equivalent to solving
\begin{equation}\label{eq:quadratic}
  \frac{\Delta^2}{2+\Delta}
  =
  c,
\end{equation}
where
\begin{equation}\label{eq:c}
  c
  :=
  \frac
  {\ln\frac{2}{\epsilon}}
  {N_1 \delta + \frac{N_1(N_1+1)}{2N_0}}.
\end{equation}
Solving the quadratic equation~\eqref{eq:quadratic} gives
\[
  \Delta
  =
  \frac{c+\sqrt{c^2+8c}}{2}
  \le
  \sqrt{c^2+8c}
  \le
  c+4.
\]
We set, conservatively, $\Delta:=c+4$.
According to \eqref{eq:A},
we have
\[
  A
  <
  (1+\Delta)
  \left(
    N_1 \delta + \frac{N_1(N_1+1)}{2N_0}
  \right)
  =
  (c+5)
  \left(
    N_1 \delta + \frac{N_1(N_1+1)}{2N_0}
  \right)
\]
with probability at least $1-\frac{\epsilon}{2}$.
It remains to plug \eqref{eq:c} and then \eqref{eq:delta} into the last inequality.

\section{Martingale multiplicative Chernoff bound}
\label{app:MMM-Chernoff}

This result was stated and proved in \cite[Corollary~6]{Kuszmaul/Qi:arXiv2102};
we have used it in Appendix~\ref{app:proof}.

\begin{theorem}\label{thm:MMM}
  Fix a filtration $(\FFF_n)$.
  Suppose $(\xi_n)$ is a (super)martingale difference
  such that $\xi_n$ take values in $\{0,1\}$,
  and set $\theta_n:=\P(\xi_n=1\mid\FFF_{n-1})$.
  Then, for any constant $s>0$,
  \begin{equation}\label{eq:S}
    S_n
    :=
    \prod_{i=1}^n
    \exp
    \left(
      s\xi_i - \theta_i(e^s-1)
    \right)
  \end{equation}
  is a test supermartingale.
  For any $\delta>0$, any $N$, and any constant $\mu$
  such that $\sum_{i=1}^N\theta_i\le\mu$ a.s.,
  \begin{equation}\label{eq:MMM}
    \P
    \left(
      \sum_{i=1}^N \xi_i
      \ge
      (1+\delta) \mu
    \right)
    \le
    \left(
      \frac{e^{\delta}}{(1+\delta)^{1+\delta}}
    \right)^{\mu}
    \le
    \exp
    \left(
      -\frac{\delta^2}{2+\delta}\mu
    \right).
  \end{equation}
\end{theorem}

\noindent
Theorem~\ref{thm:MMM} can be easily stated in terms of game-theoretic probability \cite{Shafer/Vovk:2019},
but in this paper we prefer the more familiar measure-theoretic statement.
Notice that the left-hand side of~\eqref{eq:MMM} can be replaced by
\begin{equation*}
  \P
  \left(
    \exists n\in\{1,\dots,N\}:
    \sum_{i=1}^n \xi_i
    \ge
    (1+\delta) \mu
  \right)
\end{equation*}
as $\xi_i\ge0$.
For completeness, we will give a simple proof of Theorem~\ref{thm:MMM}.

\begin{proof}[Proof of Theorem~\ref{thm:MMM}]
  The proof is standard (see, e.g., \cite{Goemans:2015} for a nice exposition).
  To check that $(S_n)$ is a supermartingale
  it suffices to notice that (assuming $(\xi_n)$ is a martingale)
  \begin{multline*}
    \E
    \left(
      \exp
      \left(
        s\xi_n - \theta_n(e^s-1)
      \right)
      \mid
      \FFF_{n-1}
    \right)
    =
    \left(
      \theta_n e^s + 1 - \theta_n
    \right)
    \exp
    \left(
      -\theta_n(e^s-1)
    \right)\\
    \le
    \exp
    \left(
      \theta_n(e^s-1)
    \right)
    \exp
    \left(
      -\theta_n(e^s-1)
    \right)
    =
    1,
  \end{multline*}
  where the inequality ``$\le$'' follows from $1+x\le e^x$.

  This is a derivation of \eqref{eq:MMM}:
  \begin{align}
    &\P
    \left(
      \sum_{i=1}^N \xi_i
      \ge
      (1+\delta)\mu
    \right)
    =
    \P
    \left(
      \exp
      \left(
        s \sum_{i=1}^N \xi_i
      \right)
      \ge
      \exp
      \left(
        s(1+\delta)\mu
      \right)
    \right)\notag\\
    &\le
    \P
    \left(
      \exp
      \left(
        s \sum_{i=1}^N \xi_i
        -
        (e^s-1) \sum_{i=1}^N\theta_i
      \right)
      \ge
      \exp
      \left(
        s(1+\delta)\mu
        -
        (e^s-1)\mu
      \right)
    \right)\label{eq:1}\\
    &\le
    \exp
    \Bigl(
      -s(1+\delta)\mu
      +
      (e^s-1)\mu
    \Bigr)
    =
    \left(
      \frac{e^{\delta}}{(1+\delta)^{1+\delta}}
    \right)^{\mu}
    \le
    \exp
    \left(
      -\frac{\delta^2}{2+\delta}\mu
    \right),
    \label{eq:2}
  \end{align}
  where $s:=\ln(1+\delta)>0$.
  The inequality ``$\le$'' in \eqref{eq:1}
  follows from the condition $\sum_{i=1}^N\theta_i\le\mu$ a.s.
  The first inequality ``$\le$'' in \eqref{eq:2} follows from \eqref{eq:S}
  being a test supermartingale.
  The equality in \eqref{eq:2} follows from $s:=\ln(1+\delta)$
  (which we obtain by minimizing the expression to the left of the ``$=$'' in $s$).
  The second inequality ``$\le$'' in \eqref{eq:2} follows from
  \begin{equation}\label{eq:ln}
    \ln(1+\delta)
    \ge
    \frac{\delta}{1+\frac{\delta}{2}},
    \quad
    \delta>0,
  \end{equation}
  applied to the logarithm of the expression on its left-hand side
  (\eqref{eq:ln} can be proved by, e.g., differentiating both sides).
\end{proof}

\begin{thebibliography}{10}
\bibitem{Basseville/Nikiforov:1993}
Mich\`ele Basseville and Igor~V. Nikiforov.
\newblock {\em Detection of Abrupt Changes: Theory and Application}.
\newblock Prentice-Hall, Englewood Cliffs, NJ, 1993.

\bibitem{Follmer/Schied:2016}
Hans F\"ollmer and Alexander Schied.
\newblock {\em Stochastic Finance: An Introduction in Discrete Time}.
\newblock De Gruyter, Berlin, fourth edition, 2016.

\bibitem{Goel/Wu:1971}
Amrit~L. Goel and S.~M. Wu.
\newblock Determination of {A.~R.~L.}\ and a contour nomogram for {Cusum}
  charts to control normal mean.
\newblock {\em Technometrics}, 13:221--230, 1971.

\bibitem{Goemans:2015}
Michel Goemans.
\newblock Chernoff bounds, and some applications: Lecture notes.
\newblock Principles of Discrete Applied Mathematics, Massachusetts Institute
  of Technology, February 2015.

\bibitem{Kuszmaul/Qi:arXiv2102}
William Kuszmaul and Qi~Qi.
\newblock The multiplicative version of {Azuma's} inequality, with an
  application to contention analysis.
\newblock Technical Report
  \href{https://arxiv.org/abs/2102.05077}{arXiv:2102.05077 [cs.DS]},
  \href{https://arxiv.org}{arXiv.org} e-Print archive, February 2021.

\bibitem{Lehmann/Romano:2022}
Erich~L. Lehmann and Joseph~P. Romano.
\newblock {\em Testing Statistical Hypotheses}.
\newblock Springer, Cham, fourth edition, 2022.

\bibitem{Lorden:1971}
Gary Lorden.
\newblock Procedures for reacting to a change in distribution.
\newblock {\em Annals of Mathematical Statistics}, 42:1897--1908, 1971.

\bibitem{Moustakides:1986}
George~V. Moustakides.
\newblock Optimal stopping times for detecting a change in distribution.
\newblock {\em Annals of Statistics}, 14:1379--1388, 1986.

\bibitem{Page:1954}
Ewan~S. Page.
\newblock Continuous inspection schemes.
\newblock {\em Biometrika}, 41:100--115, 1954.

\bibitem{Poor/Hadjiliadis:2009}
H.~Vincent Poor and Olympia Hadjiliadis.
\newblock {\em Quickest Detection}.
\newblock Cambridge University Press, Cambridge, 2009.

\bibitem{Ritov:1990}
Ya'acov Ritov.
\newblock Decision theoretic optimality of the {CUSUM} procedure.
\newblock {\em Annals of Statistics}, 18:1464--1469, 1990.

\bibitem{Shafer/Vovk:2019}
Glenn Shafer and Vladimir Vovk.
\newblock {\em Game-The\-o\-ret\-ic Foundations for Probability and Finance}.
\newblock Wiley, Hoboken, NJ, 2019.

\bibitem{Tartakovsky/etal:2015}
Alexander Tartakovsky, Igor Nikiforov, and Mich\`ele Basseville.
\newblock {\em Sequential Analysis: Hypothesis Testing and Changepoint
  Detection}.
\newblock CRC Press, Boca Raton, FL, 2015.

\bibitem{Volkhonskiy/etal:2017COPA}
Denis Volkhonskiy, Evgeny Burnaev, Ilia Nouretdinov, Alex Gammerman, and
  Vladimir Vovk.
\newblock Inductive conformal martingales for change-point detection.
\newblock {\em Proceedings of Machine Learning Research}, 60:132--153, 2017.
\newblock {COPA} 2017.

\bibitem{Vovk/etal:2022book}
Vladimir Vovk, Alex Gammerman, and Glenn Shafer.
\newblock {\em Algorithmic Learning in a Random World}.
\newblock Springer, Cham, second edition, 2022.

\bibitem{Vovk/etal:arXiv2006v2}
Vladimir Vovk, Ilia Nouretdinov, and Alex Gammerman.
\newblock Conformal e-test\-/ing.
\newblock Technical Report
  \href{https://arxiv.org/abs/2006.02329v2}{arXiv:2006.02329v2 [math.ST]},
  \href{https://arxiv.org}{arXiv.org} e-Print archive, November 2024.
\end{thebibliography}
\end{document}